\documentclass{article}
\usepackage{arxiv}
\usepackage[T1]{fontenc}
\usepackage{lmodern}
\usepackage{amsmath, amsfonts, amssymb, geometry, color, pgf, tikz, graphicx, yhmath, bm, esvect, framed, autobreak, tabularray, amsthm, caption}

\newcommand{\pa}{\ \, /\kern -0.75em / \ \,} 

\makeatletter

\newcommand{\Rmnum}[1]{\expandafter\@slowromancap\romannumeral #1@}
\makeatother

\newtheorem{theorem}{Theorem}
\newtheorem{corollary}{Corollary}[theorem]

\begin{document}

\title{A Metric Framework for Triangle Inequalities via Barycentric Coordinates}
\markright{Submission}
\author{Xi Feng \\
	Beijing National Day School\\
        Beijing 100039, P. R. China\\
	\texttt{2101110604@stu.pku.edu.cn} \\
}
\maketitle

\begin{abstract}
This paper presents a unified metric-based framework for triangle geometric inequalities using barycentric coordinates. By interpreting classical inequalities as squared distances between points—a process termed metricization—we derive and refine numerous well-known inequalities. Furthermore, the squared-distance function ${\rm ELD}_I$, measures distances from the incenter to points on the Euler line, also yielding generalized inequalities. The method offers geometric clarity and extends naturally to higher-dimensional and non-Euclidean settings.
\end{abstract}
\section{Introduction}
\noindent The theory of geometric inequalities pertaining to triangles constitutes a profoundly fertile area of mathematical inquiry. The seminal work by Mitrinović \textit{et al}. \cite{Mitrinovic1989} offers a systematic taxonomy of diverse classes of geometric inequalities and consolidates the most significant advances known at the time. In recent years, scholarly attention has increasingly focused on various dimensions of this field, including the refinement and strengthening of classical inequalities \cite{Lukarevski2019}, the exploration of underlying structural interrelations among established results (\cite{Lukarevski20232, Lukarevski20233}), and the employment of computer-assisted techniques to discover novel inequalities \cite{Rabinowitz2022}. Despite these developments, the existing body of research remains largely fragmented, lacking a cohesive methodological framework and a unified theoretical lens through which these inequalities can be comprehensively understood and systematically derived.

This paper builds upon the formulation in \cite{Feng2025}, which introduces a barycentric coordinate-based method for interpreting classical geometric inequalities as squared distances between points. Through this metric representation—termed metricization—many well-known inequalities, including those of Wolstenholme, Garfunkel–Bankoff, Kooi, Oppenheim, Neuberg–Pedoe, Klamkin, and Weitzenböck, are unified and strengthened. By selecting appropriate points in the barycentric coordinates, we transform each inequality into a non-negative squared length, thereby revealing a common geometric structure behind seemingly disparate results. This approach not only provides new proofs but also offers deeper insights into the nature and symmetry of classical triangle inequalities.
\section{An introduction to the barycentric coordinates}
In this section, we summarize the fundamental framework of barycentric coordinates as developed in \cite{Feng2025}. Let $\triangle ABC$ be a fixed nondegenerate reference triangle. The side lengths of the $\triangle ABC$ are denoted by $a = |BC|,\,\,b = |CA|$ and $c = |AB|$, The Conway notation is introduced as follows: 
\begin{align*}
    &S_A=\frac 12\left(b^2+c^2-a^2\right),
   &&S_B=\frac 12\left(c^2+a^2-b^2\right),
   &&S_C=\frac 12\left(a^2+b^2-c^2\right).
\end{align*}
For any point $P$, its barycentric coordinate is represented as a row vector:
\begin{align*}
        P = \frac{1}{S_{\triangle ABC}}[
S_{\triangle PBC},\,S_{\triangle APC},\,S_{\triangle ABP}],
\end{align*}
where $S_{\triangle XYZ}$ represents the oriented area of the triangle $XYZ$, with the convention that the area is positive if the $X,\, Y, \,Z$ are arranged counterclockwise. Moreover, the oriented area can be expressed via the determinant involving the barycentric coordinates of the points:
\begin{align*}
    S_{\triangle XYZ}=2S[X;Y;Z],
\end{align*}
where $S=S_{\triangle ABC}$ denotes the oriented area of the reference triangle $ABC$. 
For any four points $P,\,Q,\,M,\,N$, the scalar product of vectors $\vv{PQ}$ and $\vv{MN}$ satisfies
\begin{align*}
      \vv{PQ}\cdot \vv{MN}&=\left(Q-P\right)K\left(N-M\right)^T,
\end{align*}
where the metric matrix $K$ is given by
\begin{align*}
    K=\begin{bmatrix}
S_A & 0  & 0 \\
0  & S_B & 0 \\
0  & 0  & S_C \\
\end{bmatrix}.
\end{align*}
In particular, the squared distance between two points $P$ and $Q$ can be computed as
\begin{align*}
    |PQ|^2=PKP^T-2PKQ^T+QKQ^T.
\end{align*}
For the orthocenter $H$ of the triangle $ABC$ and an arbitrary point $X$, the following H kernel property holds:
\begin{align*}
    HKX^T=HKX^T.
\end{align*}
As a direct consequence, we obtain the inequality:
\begin{align}
    |HX|^2=XKX^T-HKH^T\geq0.
\end{align}
Furthermore, for any two points $P,\,Q$ and the point at infinity $l$, the following identity holds:
\begin{align*}
    |PQ|^2=
    \frac{1}{|l|^2}\left[\left(P-Q\right)Kl^T\right]^2+
    \frac{4S^2}{|l|^2}\left[P;Q;l\right]^2.
\end{align*}
Obviously all terms on the righ-hand side are nonnegative, we immediately derive the following theorem:
\begin{theorem}
    \begin{align}
        |PQ|^2&\geq
    \frac{1}{|l|^2}\left[\left(P-Q\right)Kl^T\right]^2,\label{Cauchy}\\
    |PQ|^2&\geq
    \frac{4S^2}{|l|^2}\left[P;Q;l\right]^2.
    \label{Cauchy 2}
    \end{align}
\end{theorem}
Notably, the inequality (\ref{Cauchy}) can be regared as a Cauchy-Schwarz-type inequality in barycentric coordinates. This framework allows one to transform a given inequality into the form $F\geq0$, then interpret $F$ as a squared distance between two points, thereby establishing the nonnegativity naturally.
Moreover, by selecting different points at infinity $l$, the inequalities (\ref{Cauchy}) and (\ref{Cauchy 2}) provide a mechanism to derive various strengthened forms of the original inequality. We refer to this approach as the metricization of inequalities.
\section{Metricization of some classic inequalities}
Let 
\begin{align*}
    &P = \left(x+y+z\right)^{-1}\left[y,\,z,\,x\right],
    &&Q = 
    \left(x+y+z\right)^{-1}\left[z,\,x,\,y\right],
\end{align*}
then the following identity holds \cite{Oppenheim1964}:
\begin{align}
    a^2(x-y)(x-z)+b^2(y-x)(y-z)+c^2(z-x)(z-y)=(x+y+z)^2|PQ|^2\geq0.
\end{align}
This expression can be interpreted as the squared distance between two moving points $P$ and $Q$. By choosing
\begin{align*}
    &P=\left(\frac xa+\frac yb+\frac zc\right)^{-1}
    \left[\frac yb,\, \frac zc,\, \frac xa\right],
    &&Q=\left(\frac xa+\frac yb+\frac zc\right)^{-1}
    \left[\frac zc,\, \frac xa, \,\frac yb\right],
\end{align*}
we ontain:
\begin{align}
  {\rm Wolstenholme}=
  \left(\frac xa+\frac yb+\frac zc\right)^2|PQ|^2\geq0,
\end{align}
where
\begin{align*}
    {\rm Wolstenholme}=x^2+y^2+z^2-2yz\cos A-2zx\cos B-2xy\cos C,
\end{align*}
which is known as the Wolstenholme inequality \cite{Wolstenholme1867}. It serves as a basis for deriving a variety of geometric inequalities. For instance, if we choose $x=\tan \frac A2,\,y=\tan \frac B2,\,z=\tan \frac C2$, then
\begin{align*}
    &P=\mu^{-1}\left[
\cos^{-2}\frac B2,\,
\cos^{-2}\frac C2,\,
\cos^{-2}\frac A2 \right],&&
Q=\mu^{-1}\left[
\cos^{-2}\frac C2,\,\cos^{-2}\frac A2,\,
\cos^{-2}\frac B2\right],
\end{align*}
where
$\mu = \cos^{-2}\frac A2+
\cos^{-2}\frac B2+
\cos^{-2}\frac C2$. This leads to:
\begin{align}
    {\rm Garfunkel\text{-}Bankoff}=\frac{\mu^2}{16R^2}|PQ|^2\geq 0,
\end{align}
where
\begin{align*}
    {\rm Garfunkel\text{-}Bankoff}=\tan^2\frac A2+\tan^2\frac B2+\tan^2\frac C2-
    \left(2-8\sin\frac A2\sin\frac B2\sin\frac C2\right),
\end{align*}
which is known as the Garfunkel-Bankoff inequality (\cite{Garfunkel1983, Bankoff1984}). Now consider the triangle $\triangle DEF$ formed by the tangents to the circumcircle of triangle $\triangle ABC$ at points $A$, $B$, and $C$, as shown in Figure \ref{fig:1}. Since the point $D$ lies on the perpendicular bisector of segment $BC$, its coordinates
can be expressed as:
\begin{align*}
    D = \frac 12\left(B+C\right)+\lambda\left(H-A\right),
\end{align*}
where $\lambda$ is determined by the orthogonality condition $(O-B)K(D-B)^T=0$. Consequently, points $D,\,E,\,F$ can be expressed as:
\begin{align*}
D &= \frac 12\left(B+C\right)+\frac{a^2}{2\left(S_A-HKH^T\right)}\left(H-A\right),\\
E &= \frac 12\left(C+A\right)+\frac{b^2}{2\left(S_B-HKH^T\right)}\left(H-B\right),\\
F &= \frac 12\left(A+B\right)+\frac{c^2}{2\left(S_C-HKH^T\right)}\left(H-C\right).
\end{align*}
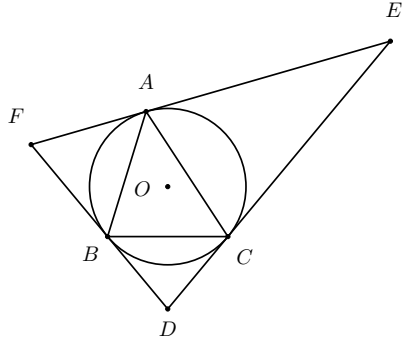
\begin{figure}
    \centering
    \scalebox{0.8}{
    \begin{tikzpicture}[line cap=round,line join=round,x=1.0cm,y=1.0cm]
\clip(-4.3,-2.16) rectangle (7.1,6.3);
\draw [line width=0.8pt] (-0.36,3.08)-- (-1.,1.);
\draw [line width=0.8pt] (-1.,1.)-- (1.,1.);
\draw [line width=0.8pt] (1.,1.)-- (-0.36,3.08);
\draw [line width=0.8pt] (0.,1.8307692307692307) circle (1.3000682731275692cm);
\draw [line width=0.8pt] (-2.270588235294118,2.5294117647058822)-- (3.7,4.25);
\draw [line width=0.8pt] (3.7,4.25)-- (0.,-0.2037037037037046);
\draw [line width=0.8pt] (0.,-0.2037037037037046)-- (-2.270588235294118,2.5294117647058822);
\draw [fill=black] (-1.,1.) circle (1.0pt);
\draw[color=black] (-1.28,0.71) node {$B$};
\draw [fill=black] (1.,1.) circle (1.0pt);
\draw[color=black] (1.28,0.67) node {$C$};
\draw [fill=black] (-0.36,3.08) circle (1.0pt);
\draw[color=black] (-0.36,3.57) node {$A$};
\draw [fill=black] (0.,1.8307692307692307) circle (1.0pt);
\draw[color=black] (-0.42,1.79) node {$O$};
\draw [fill=black] (0.,-0.2037037037037046) circle (1.0pt);
\draw[color=black] (0.,-0.53) node {$D$};
\draw [fill=black] (3.7,4.25) circle (1.0pt);
\draw[color=black] (3.76,4.77) node {$E$};
\draw [fill=black] (-2.270588235294118,2.5294117647058822) circle (1.0pt);
\draw[color=black] (-2.52,3.01) node {$F$};
\end{tikzpicture}
}
    \caption{The $\triangle DEF$ is formed by the tangents to the circumcircle of the $\triangle ABC$ at points $A$, $B$, and $C$.}
    \label{fig:1}
\end{figure}
Let
\begin{align*}
    &P = \mu^{-1}\left(\frac y{|FD|}D+\frac z{|DE|}E+\frac x{|EF|} F\right), 
    &&Q = \mu^{-1}\left(\frac z{|DE|}D+\frac x{|EF|} E+\frac y{|FD|}F\right),
\end{align*}
with $\mu = \frac x{|EF|}+\frac y{|FD|}+\frac z{|DE|}$. Then the following identity holds:
\begin{align}
    {\rm Kooi}=\mu^2R^2|PQ|^2\geq 0,
\end{align}
where
\begin{align*}
    {\rm Kooi}=R^2\left(x+y+z\right)^2-\left(
    yza^2+zxb^2+xyc^2\right),
\end{align*}
which is precisely the Kooi's inequality \cite{Kooi1958}. When the $\triangle ABC$ is obtuse, certain side lengths should be treated as negative. Specifically:
\begin{align*}
    &|EF| = \frac{a^3bc}{2S_BS_C},&&
    |FD| = \frac{ab^3c}{2S_AS_C},&&
    |DE| = \frac{abc^3}{2S_AS_B}.
\end{align*}
Upon substituting $(x,y,z)$ with $(a^2x,\,b^2y,\,c^2z)$, we obtain:
\begin{align*}
    &P = \mu^{-1}\left(\frac{yb^2}{|FD|}D+\frac{zc^2}{|DE|}E+\frac{xa^2}{|EF|} F\right), 
    &&Q = \mu^{-1}\left(\frac{zc^2}{|DE|}D+\frac{xa^2}{|EF|} E+\frac {yb^2}{|FD|}F\right),
\end{align*}
with $\mu=\frac{a^2x}{|EF|}+\frac{b^2y}{|FD|}+\frac{c^2z}{|DE|}$, and the identity bocomes:
\begin{align}
    {\rm Oppenheim}=\mu^2|PQ|^2\geq 0,
\end{align}
where
\begin{align*}
    {\rm Oppenheim}= \left(xa^2+yb^2+zc^2\right)^2-16S^2\left(yz+zx+xy\right),
\end{align*}
which is the Oppenheim inequality \cite{Oppenheim1961}. Furthermore, by choosing $x=S'_A,\,y=S'_B,\,z = S'_C$ -- the Conway notation for another triangle $\triangle A'B'C'$ -- we obtain the Neuberg–Pedoe's inequality \cite{Pedoe1943}.

To provide another perspective, we reinterpret these inequalities via point-distance forms. Let $P=(x+y+z)^{-1}[x,\,y,\,z]$, and let $Q$ be any fixed point in the plane, then we have
\begin{align*}
    x|QA|^2+y|QB|^2+z|QC|^2&=
    \left(x+y+z\right)\left(QKQ^T-2PKQ^T+3GKP^T\right),\\
    yza^2+zxb^2+xyc^2&=-\left(x+y+z\right)^2
    \left(PKP^T-3GKP^T\right),
\end{align*}
where $G=\frac13[1,1,1]$ represents the centroid of the triangle $ABC$. It follows that:
\begin{align}
    {\rm Klamkin}=\left(x+y+z\right)^2|PQ|^2\geq0,
\end{align}\label{Klamkin inequality}
where
\begin{align*}
    {\rm Klamkin}=\left(x+y+z\right)
    \left(x|QA|^2+y|QB|^2+z|QC|^2\right)-\left(
    yza^2+zxb^2+xyc^2\right).
\end{align*}
The inequality (\ref{Klamkin inequality}) is known as the polar moment of inertia inequality \cite{Klamkin1975}.   
Let $Q=O$, the circumcenter $O$ of the triangle $ABC$, this reduces to Kooi's inequality:
\begin{align}
    {\rm Kooi}=\left(x+y+z\right)^2|PO|^2\geq0.
\end{align}
In particular, when $P$ is the symmedian point $X_6$ (denoted by $X(6)$ in Encyclopedia of Triangle Centers (\cite{kimberling_etc}, hereafter referred to as ETC), we have:
\begin{align}
    {\rm Weizenbock}\cdot{\rm Weizenbock}^\star=\frac{a^2+b^2+c^2}{R^2}|OX_6|^2\geq0,
\end{align}
where
\begin{align*}
    &{\rm Weizenbock}=a^2+b^2+c^2-4\sqrt3S,&&&{\rm Weizenbock}^\star=a^2+b^2+c^2+4\sqrt3S.
\end{align*}
which is the Weizenbock inequality \cite{Weitzenboeck1919}.
Letting $P = \left(a^2x+b^2y+c^2z\right)^{-1}\left[a^2x,\,b^2y,\,c^2z\right]$, the identity is transferred to Oppenheim's inequality as follows:
\begin{align}
    {\rm Oppenheim}=\frac1{R^2}\left(xa^2+yb^2+zc^2\right)^2|PO|^2\geq0.
\end{align}
Let $D_1,\,E_1,\,F_1$ be the points where the incircle touches the sides of the triangle; specifically,
\begin{align*}
    &D_1 = \frac 1a\left[0,\,s-c,\,s-b\right],&&
    &&E_1 = \frac 1b\left[s-c,\,0,\,s-a\right],
    &&F_1 = \frac 1c\left[s-b,\,s-a,\,0\right].
\end{align*}
Letting $P = (x+y+z)^{-1}(xD_1+yE_1+zF_1)$, we obtain (\cite{Lukarevski2023}):
\begin{align}
    {\rm Wolstenholme}=\frac{\left(x+y+z\right)^2}{r^2}|PI|^2.
\end{align}
In particular, set
\begin{align*}
    P=\mu^{-1}\left(\tan\frac A2\cdot D_1+\tan \frac B2\cdot E_1+\tan \frac C2\cdot F_1\right),\quad\text{where}\quad \mu=\tan\frac A2+\tan \frac B2+\tan \frac C2,
\end{align*}
then
\begin{align}
    {\rm Garfunkel\text{-}Bankoff}=\frac{\mu^2}{r^2}|PI|^2\geq0.
\end{align}
\section{Distance from incenter to Euler line point}
This section is devoted to the study of concrete geometric inequalities via metric analysis. As is customary, we express the fundamental elements of a triangle in terms of the semiperimeter $s$, the inradius $r$, and the circumradius $R$. The function (${\rm ELD}_I$) is defined to represent the squared distance from the incenter $I$ to a point $X$ on the Euler line:
\begin{align*}
    {\rm ELD}_I\left(\lambda\right)&=|XI|^2
    =|OH|^2\lambda^2-2\left(OKI^T-HKH^T\right)\lambda+|IH|^2,
\end{align*}
where $X=\left(1-\lambda\right)H+\lambda O$, with the following coefficients:
\begin{align*}
    |OH|^2&=9R^2+8Rr+2r^2-2s^2,\\
    OKI^T-HKH^T&=12R^2+14Rr+5r^2-3s^2,\\
    |IH|^2&=4R^2+4Rr+3r^2-s^2.
\end{align*}
The minimum value of ${\rm ELD}_I\left(\lambda\right)$,
denoted $d_I$, represents the squared distance from the incenter $I$ to the Euler line: 
\begin{align*}
    d_I=\frac{4S^2\left[H;I;O\right]^2}{|OH|^2}=
    -\frac{s^4-2\left(2R^2+10Rr-r^2\right)s^2+r(4R+r)^3}{4\left(9R^2+8Rr+2r^2-2s^2\right)}.
\end{align*}
In particular \cite{Mitrinovic1989},
\begin{align}
    s^4-2\left(2R^2+10Rr-r^2\right)s^2+r\left(4R+r\right)^3
    =-16S^2\left[H;I;O\right]^2\leq 0.
    \label{Bottema 1}
\end{align}
Treating LHS as a quadratic in $s^2$, let the roots be $s^2_1$ and $s^2_2$. Then we have \cite{Blundon1965}:
\begin{align}
    s^2_1\leq s^2\leq s^2_2,\quad\text{where}\quad    s^2_{1,2}=2R^2+10Rr-r^2\mp 
    2\left(R-2r\right)\sqrt{R\left(R-2r\right)}.
    \label{Bottema 2}
\end{align}
Moreover, the squared distance from an any point $P$ ($P = \alpha I+\beta O+ \gamma H$) to the Euler line can be expressed as:
\begin{align*}
    d_P=\frac{4S^2\left[H;P;O\right]^2}{|OH|^2}=\alpha^2d_I.
\end{align*}
A fundamental result can now be stated as follows:
\begin{theorem}[ELD inequality]
    \begin{align}
        {\rm ELD}_I\left(\lambda\right)\geq d_I\geq0.
        \label{ELD inequality}
    \end{align}
\end{theorem}
Several interesting consequences follow immediately:
\begin{corollary}
In any non-degenerate triangle, the angle $\angle HIO$ is obtuse.
\end{corollary}
\begin{proof}
\begin{align*}
    2\left(I-H\right)K\left(I-O\right)^T=-{\rm ELD}_I(0)-\frac{2r}R{\rm ELD}_I(1)<0.
\end{align*}
\end{proof}
\begin{corollary}
   \begin{align*}
       a^2+b^2+c^2\leq 9R^2.
   \end{align*}
\end{corollary}
\begin{proof}
     \begin{align*}
         9R^2-a^2-b^2-c^2=2{\rm ELD}_I\left(0\right)+\left(1+\frac{2r}R\right){\rm ELD}_I\left(1\right)\geq0.
     \end{align*}
\end{proof}
\begin{corollary}
    \begin{align*}
        \frac1{a^2}+\frac1{b^2}+\frac1{c^2}\leq\frac{1}{4r^2}.
    \end{align*}
\end{corollary}
\begin{proof}
    \begin{align*}
        \left(2Rrs\right)^2\left(\frac{1}{4r^2}-\frac1{a^2}-\frac1{b^2}-\frac1{c^2}\right)=|OH|^2\cdot d_I+Rr\cdot{\rm ELD}_I\left(2\right)\geq0.
    \end{align*}
\end{proof}
Several classical inequalities can also be represented in terms of ${\rm ELD}_I\left(\lambda\right)$: 
\begin{align}
    R-2r&=\frac 1R{\rm ELD}_I\left(1\right),
    \label{Euler}\\
    4R^2+4Rr+3r^2-s^2&={\rm ELD}_I\left(0\right),
    \label{Gerretsen1}\\
    s^2-r\left(16R-5r\right)&=9{\rm ELD}_I\left(\frac 23\right),
    \label{Gerretsen2}\\
    4R+r-\sqrt{3}s&=\frac{{\rm ELD}_I\left(2\right)}{4R+r+\sqrt{3}s}.\label{FinslerHadwiger}
\end{align}
The equation (\ref{Euler}) is Euler's inequality, the equations (\ref{Gerretsen1}) and (\ref{Gerretsen2}) correspond to Gerretsen's inequality {\cite{Gerretsen1953}}, and the equation (\ref{FinslerHadwiger}) is equivalent to
the Finsler-Hadwiger's inequality {\cite{FinslerHadwiger1937}}, as shown by the identity:
\begin{align*}
    a^2+b^2+c^2-(a-b)^2-(b-c)^2-(c-a)^2-4\sqrt3S=4r\left(4R+r-\sqrt3s\right).
\end{align*}
By the ELD inequality, the above classical inequalities can be sharpened as follows:
\begin{theorem}[Sharpened Classical Inequalities]
    \begin{align}
    R -2r&\geq \frac{d_I}{R},\\
    4R^2+4Rr+3r^2-s^2 &\geq d_I,\\
    s^2-r\left(16R-5r\right)&\geq 9d_I,\\
    4R+r-\sqrt{3}s&\geq\frac{d_I}{4R+r+\sqrt{3}s}.
\end{align}
\end{theorem}
Moreover, the axis of symmetry of ${\rm ELD}_I(\lambda)$ lies within a specific interval:
\begin{theorem}
    The axis of symmetry of ${\rm ELD}_I(\lambda)$ lies in the interval $\left[0,\frac 23\right]$.
\end{theorem}
\begin{proof}
Using the identities:
    \begin{align*}
    2\left(OKI^T-HKH^T\right)&=3{\rm ELD}_I\left(0\right)+
    \frac{2r}{R}{\rm ELD}_I\left(1\right)\geq 0,\\
    2\left(OKO^T-HKH^T\right)-3\left(OKI^T-HKH^T\right)&=
    \frac 92 {\rm ELD}_I\left(\frac 23\right)+\frac{3r}{R}{\rm ELD}_I\left(1\right)\geq 0,
\end{align*}
we conclude the axis lies at
$\lambda = \frac{OKI^T-HKH^T}{OKO^T-HKH^T}\in\left[0,\frac 23\right]$. 
\end{proof}
As an immediate consequence:
\begin{align*}
    {\rm ELD}_I\left(2\right)\geq{\rm ELD}_I\left(1\right)\geq{\rm ELD}_I\left(\frac 23\right).
\end{align*}
We now apply the method to sharpen the Kooi's inequality \cite{Kooi1958}:
\begin{align*}
    {\rm Kooi}^{\star} = \frac{R\left(4R+r\right)^2}{2\left(2R-r\right)}-s^2\geq0.
\end{align*}
This can be reformulated as a metric expression \cite{Lukarevski2019}:
\begin{align*}
    {\rm Kooi}^{\star}=\frac{\left(4R+r\right)^2}{2R\left(2R-r\right)}|OM|^2,
\end{align*}
where $M$ is the Mittenpunkt ($X(9)$ in ETC).
The squared distance from $M$ to a point $X$ ($X = H+\lambda\left(O-H\right)$) on the Euler line is:
\begin{align}
    {\rm ELD}_M(\lambda)=|MX|^2=|OH|^2\lambda^2+
    2\left(O-H\right)K\left(H-M\right)^T\lambda+|HM|^2,
\end{align}
with the coefficients:
\begin{align*}
    |OH|^2&=9R^2+8Rr+2r^2-2s^2,\\
    |HM|^2&=\frac{\left(2R+r\right)^2}{\left(4R+r\right)^2}\left(16R^2+8Rr+r^2-3s^2\right),\\
    2\left(O-H\right)K\left(H-M\right)^T&=\frac{2R+r}{4R+r}\left(5s^2-24R^2-18Rr-3r^2\right).
\end{align*}
The minimum of ${\rm ELD}_M(\lambda)$ is
\begin{align*}
    \frac{4S^2|O;H;M|^2}{|OH|^2}=\frac{\left(2R+r\right)^2}{\left(4R+r\right)^2}d_0.
\end{align*}
Hence, Kooi's inequality can be sharpened as:
\begin{align}
    {\rm Kooi}^{\star}\geq\frac{\left(4R+r\right)^2}{2R\left(2R-r\right)}\cdot\frac{4S^2|O;H;M|^2}{|OH|^2}=\frac{\left(2R+r\right)^2d_0}{2R\left(2R-r\right)}.
    \label{Kooi sharpen}
\end{align}
An alternative form, due to \cite{Bilchev1991}, is:
\begin{align*}
    {\rm Kooi}^{\star}=\frac{\left(4R+r\right)^2}{8R\left(2R-r\right)}|HX_7|^2,
\end{align*}
where $X_7$ denotes the Gergonne point ($X(7)$ in ETC).
The same refinement method applies, yielding the same result as in (\ref{Kooi sharpen}). 
\section{Conclusion}
In this paper, we have established proofs and refinements of several classical geometric inequalities by interpreting them as distance inequalities between points in the barycentric coordinate system. By introducing the squared-distance function (${\rm ELD}_I(\lambda)$), which measures the distance from the incenter $I$ to a point on the Euler line, we derived geometric interpretations for a broad class of inequalities involving the classical triangle elements $s,\,R$ and $r$, and generalized various related results. 

Compared to traditional algebraic or synthetic approaches, the metric-based method not only provides deeper geometric insight into these inequalities but also offers a natural and flexible framework for extending classical results from the Euclidean plane to higher-dimensional Euclidean spaces and even to non-Euclidean geometries. Further developments along these lines will be explored in subsequent work.

\end{document}